\documentclass{amsart}

\usepackage{amsmath, amsthm, amssymb, mathtools}
\usepackage[backref=page, colorlinks, linkcolor=blue]{hyperref}

\usepackage{bbm} 
\usepackage{marvosym}

\newtheorem{theorem}{Theorem}[section]

\newtheorem{corollary}[theorem]{Corollary}
\newtheorem{proposition}[theorem]{Proposition}
\newtheorem*{theoremA}{Theorem A}
\newtheorem*{theoremB}{Theorem B}
\newtheorem*{theoremC}{Theorem C}
\newtheorem*{theoremD}{Theorem D}
\theoremstyle{definition}
\newtheorem{definition}[theorem]{Definition}

\theoremstyle{remark}
\newtheorem{remark}[theorem]{Remark}

\usepackage{color}
\usepackage[latin1]{inputenc}

\begin{document}

\title[Fluctuations of ergodic sums on periodic orbits]{Fluctuations of Ergodic Sums on Periodic Orbits under Specification}

\author{Manfred Denker}
\address{The Pennsylvania State University}

\author{Samuel Senti}
\address{Universidade Federal do Rio de Janeiro}

\author{Xuan Zhang}
\address{Universidade de S\~ao Paulo}


\begin{abstract}
We study the fluctuations of ergodic sums using global and local specifications on periodic points. We obtain Lindeberg-type central limit theorems in both situations. As an application, when the system possesses a unique measure of maximal entropy, we show weak convergence of ergodic sums to a mixture of normal distributions. Our results suggest decomposing the variances of ergodic sums according to global and local sources.
\end{abstract}
\maketitle

\section{Fluctuations for ergodic sums}\label{sec:1}

The central limit theorem (CLT) for Gibbs measures was first observed by Bowen in \cite{Bow75}, although this was in fact preceded by earlier work on this subject (see \cite{Den89} for more details). Since then a vast amount of research on the CLT  in dynamical systems has appeared.
In most cases this law is shown to hold for H\"older continuous obervables or functions of bounded variation. This is due to the fact that, in the majority of cases, one can show a sufficiently rapid (summable) decay of correlations which entails  the CLT by Gordin's martingale-coboundary decomposition (\cite{Gor69}).

However, all but one of the early attempts to prove such laws in dynamics used mixing concepts which also lead to the Gaussian limit law in cases of functions with $2+\delta$ moments, which are sufficiently well approximated by an underlying probabilistic mixing structure (cf. \cite{Bra07} for basic definitions of probabilistic mixing conditions). This method allows formulating CLT's for non-continuous functions which merely need to be sufficiently well approximable (in $L^2$, or in probability) (see \cite{DenKel83} for an example of such an approach).
This mixing structure is simple when the transfer operator has a spectral gap (see \cite{Rou83} for a first result of this type), but it is not of the type considered in probability.
Extensions of the spectral gap and of the mixing methods have appeared, allowing the treatment of intermittent maps (see \cite{AarDenUrb93} for one of the first results) via tower constructions  (inducing on Darling-Kac sets, Schweiger's jump transformation, or recently  Young towers) as well as maps with slow decay of correlations.
We cite \cite{Gou04, GouMel14, HayNicVaiZha13, Liv95, Tho14, Tho15} for an incomplete list of some recent articles, including an example of a CLT for general arrays, as is presented here.

Our methods differ from the probabilistic mixing approach (see  \cite{Bra07, Dou94, IbrLin71}) as well as from the transfer operator method (see \cite{HenHer01} for a general approach).
We build upon the specification properties on periodic points, which were first introduced by Bowen (see \cite{Bow71}).
The use of periodic points has many conveniences, for example any invariant measure can be approximated by a distribution on the periodic points (see \cite{Sig70}) and periodic points can be used towards a numerical understanding of the dynamics (see \cite{DenDuaMcC09}). But it is worth noting that our results  do not depend heavily on the existence of periodic points. Specification properties using other sets of points are likely to work in the same way.
In particular, when choosing weighted point measures to approximate equilibrium states which are subject to specification, it is possible to study the fluctuation of ergodic sums  under equilibrium measures; this applies especially to maps of the interval with a Markov structure.

The classical CLTs in dynamics study the fluctuation of partial sums of  H\"older functions around their means globally using a global mixing concept (including spectral gaps), and there does not seem to be any other significant source of fluctuation.  Our results show that local randomness can also contribute to fluctuation. Hence it is important to investigate the sources of the fluctuations. We suggest a decomposition into  \textit{global}  and \textit{local} fluctuations of Ces\'aro averages and study the CLTs in this context. This description still allows for an unknown source of fluctuation which may arise from the non-uniform distribution of periodic orbits.

To set the stage, we choose to consider dynamical systems $(X,T)$ exhibiting either \textit{global} or \textit{local} specification. The specification is called global if the concatenation of any number of orbit pieces of a given length can be shadowed by a single periodic orbit, provided sufficient time (referred to as a \textit{gap}) is allowed to migrate from the end of one piece of orbit to the next (see Definition~\ref{def:spec}). The specification is called {local} when orbits that are close enough at their initial and terminal time periods can by shadowed by periodic points (see Definition~\ref{def:specloc}).

Thanks to the relation between periodic orbits given by the specification properties, one can pick out certain sets of periodic points and compare uniform measures on them with product measures to capture an independence-like structure (assuming that  there are only finitely many periodic points of any given period).
We call these \textit{globally or locally $\epsilon$-independent sets} (see definitions~\ref{def:independence} and \ref{def:locidp}).
This {global} structure allows for the study of the distribution of Birkhoff sums with gaps, i.e. incomplete time series, whereas the {local} version allows one to infer the distribution of the Birkhoff sums from the readings taken only at certain locations in the phase space.

In the global scenario, $\epsilon$-independence structure is constructed on the orbit pieces outside the gaps. To obtain a limit law it would be necessary to require certain condition on the observables over the gaps. We call this type of condition a \textit{gap condition}:
{the sum of the variances along the gaps is negligible when compared to the sum of the variances along the orbit pieces}. After neglecting all the gaps, an \textit{oscillation condition} on the observables ensures us to treat the measures on $\epsilon$-independent sets like product measures: {the first and second moments of the oscillation of the observable along the orbit pieces is negligible when compared to the sum of the variances.} Our first main theorem is a Lindeberg-type CLT for systems with global specification property (see Theorem~\ref{theo:1}).

\begin{theoremA}
	Consider a dynamical system $(X, T)$ with \emph{global} specification. Then along a sequence of $\epsilon_l$-independent sets $\mathcal P_{l}$ the following holds:
	 given observables $h_l$ satisfying the oscillation condition~\eqref{eq:4.2-0} and the gap condition~\eqref{eq:4.2-1}, if
	the Lindeberg condition~\eqref{eq:epslind} holds with respect to the uniform measure $\nu_{\mathcal P_{l}}$ on $\mathcal P_{l}$, then the central limit theorem holds:
	$$ \lim_{l\to\infty} \nu_{\mathcal P_{l}}\left(\bigg\{x\in X: \sum_{j=0}^l\Big(h_l(T^jx)-\mathbb{E}_{\nu_{\mathcal P_{l}}}(h_l)\Big)\leqslant ts_l\bigg\}\right) =\frac 1{\sqrt{2\pi}} \int_{-\infty}^t e^{-u^2/2} du$$
	where $s_l$ denotes a suitably-defined total variance. Additionally, the reverse  holds true under a uniform oscillation condition~\eqref{eq:4.2-00}.
\end{theoremA}

In the course of proving this theorem, we prove CLT for the more general \textit{dynamical arrays} (equivalently defined in~\cite[Definition 5.1]{DenSenZha17}) instead of Birkhoff sums.
\begin{definition}\label{def:arraygeral}[Dynamical Array]
	For each $l\in\mathbb N$ and $k_l\in \mathbb N$ consider two increasing sequences of positive integers
	$\{a_{l,i}\}_{1\leqslant i\leqslant k_l}$ and $\{b_{l,i}\}_{1\leqslant i\leqslant k_l}$ with
	$$
	0\leqslant a_{l,1}\leqslant b_{l,1}<a_{l,2}\leqslant b_{l,2}<\ldots< a_{l, k_l}\leqslant b_{l, k_l}
	$$
	and real-valued functions $h_{l,i}:X\to \mathbb R$ where $1\leqslant i\leqslant k_l$.
	A \emph{dynamical array} is a sequence $\{H_{l,i}\}_{ 1\leqslant i\leqslant k_l,\, l\in \mathbb N}$ of real valued functions of the form
	$$ H_{l,i} = \sum_{j=a_{l,i}}^{b_{l,i}} h_{l,i}\circ T^{j} \qquad   1\leqslant i\leqslant k_l.$$
	We refer to the intervals $a_{l,i+1}-b_{l,i}$ as the \textit{gaps} of the dynamical array.
\end{definition}
\noindent We will only consider the case $b_{l,i}-a_{l,i}=n_l$ and $a_{l,i+1}-b_{l,i}=M_l$ for all $1\leqslant i\leqslant k_l$, whereas similar results can be obtained for the more general case.

While Theorem A holds for the uniform measure supported on $\epsilon$-independent set, one can find a weighted measure supported on the set of all periodic points for which a Lindeberg-type CLT holds (see Corollary~\ref{cor:1}).
On the other hand, although the $\epsilon$-independent sets only contain incomplete periodic points, there is rich enough structure so that the entropy of any weak limit of uniform measures on these sets equals the topological entropy in an expansive system (see Theorem~\ref{theo:weak-convergence}).
We apply Theorem~A to show that it suffices to take the Birkhoff average over the full orbit of a typical periodic point to approximate the integrals of continuous functions with respect to the measure of maximal entropy (see Theorem~\ref{theo:3}).

\begin{theoremB}
Consider an expansive dynamical system with \emph{global} specification. Assume some mild additional conditions on the constants $N_l=k_l(n_l+M_l)$ defining the array and $\epsilon_l$-independent set $\mathcal P_l$, one has for a Lipschitz function $h$ and any $\eta>0$
$$\lim_{l\to\infty}\nu_{\mathcal P_{l}}\left(\bigg\{
\bigg|\frac 1{N_l}
\sum_{j=0}^{N_l}\Big(h(T^jx)-\mathbb{E}_{\nu_{\mathcal P_{l}}}(h)\Big)
\bigg|
\leqslant k_l^{-\frac 1 2+\eta}\bigg\}\right) = 1.$$
Moreover, if $(X,T)$ admits a unique measure of maximal entropy, then for random sequences of periodic points  $p_l\in\mathcal{P}_l$, the uniform distributions over the orbit of $p_l$ converge to the measure of maximal entropy.
\end{theoremB}

In the {local} scenario, local $\epsilon$-independence structure is constructed at pre-defined locations. One would not need a gap condition but an oscillation condition is still necessary. In place of Theorem A we obtain a local Lindeberg-type CLT which studies the fluctuation of partial sums around local means (see Theorem~\ref{theo:2}).

\begin{theoremC}

	Consider a dynamical system with \emph{local} specification. Then along a sequence of locally $\epsilon_l$-independent sets $\mathcal P_{l}$  the following holds:
	given a dynamical array satisfying the oscillation condition, the Lindeberg condition
	holds with respect to the uniform distribution $\nu_{\mathcal P_{l}}$ on $\mathcal P_{l}$, if and only if the array
	is $\nu_{\mathcal{P}_l}$-asymptotically negligible and the CLT holds.
\end{theoremC}

As local specification implies global specification in a topologically mixing system, one can study the fluctuation of partial sums with respect to the measure of maximal entropy by local $\epsilon$-independence structure. 
We obtain the following result (see Theorem~\ref{theo:8.2}).

\begin{theoremD}
Let $(X,T)$ be an expansive and topologically mixing dynamical system with the \emph{local} specification property. Then, with respect to  the unique measure of maximal entropy $\mu$,  the class of wildly oscillating functions in $L^3(\mu)$ (see Definition~\ref{def:5.2}) satisfying the moment condition~\eqref{eq:9.30} and with integrable local variance~\eqref{eq:9.32} belongs to the partial domain of attraction of a mixed normal distribution, i.e. a subsequence of properly centered and normed partial sums converges weakly to a mixed normal distribution.
\end{theoremD}
Since we only assume continuity for the transformation $T$, we will prove our results for positively expansive systems in this paper. Proofs are analogous for expansive homeomorphisms.

Our results clarify the amount of randomness present in a dynamical system which is due to its periodic orbit structure. From a dynamical viewpoint such theorems are important for various reasons: first, they allow one to derive CLTs for non-standard functions and for limits of invariant discrete probabilities, such as equilibrium states. In our context we only consider measures of maximal entropy but our results should also apply to more general equilibrium states. This applies in particular to Anosov diffeomorphisms and subshifts of finite type.
Secondly, such theorems provide methods for data and numerical analysis of time series. Indeed this analysis can be carried out through descriptive statistics based on dynamical arrays and their asymptotic normality. From the point of view of applications (data and numerical analysis) it is also important to study how the variance of the dynamical arrays is determined by the periodic point structure.

\textbf{Structure of the paper}
In Section \ref{sec:2} we collect basic definitions and notations.
We recall the notions of local and global specification used in the paper and illustrate them with a few examples which include
Bowen's Axiom-A$^*$-homeomorphisms, hyperbolic rational map of $\mathbb{S}^2$ and topological Markov chains with restricted entries.
In Section~\ref{sec:globalCLT}  we define the $\epsilon$-independence structure in systems with global specification and prove the CLT Theorem~A. In Section~\ref{sec:localCLT} we define the local $\epsilon$-independence in systems with local specification and prove the CLT Theorem~C. In Section~\ref{sec:mme} we apply both concepts of $\epsilon$-independence structure to study the fluctuation problem with respect to the measure of maximal entropy. We show that the uniform measures on $\epsilon$-independent sets converge weakly to the measure of maximal entropy if the latter is unique and $T$ restricted to periodic points is a homeomorphism, and prove Theorems B and D.
In Section~\ref{sec:decomposition} a decomposition of the variance is described.

\section{Notations and definitions}\label{sec:2}

Consider a continuous transformation $T:X\to X$ of a compact metric space $(X, d)$.
Denote the sets of periodic points by
$$
P_n:=\{x\in X\colon T^n x=x\} \quad\mbox{ and }\quad P:=\bigcup_{n\in\mathbb N} P_n
$$
and denote
$$
B_\epsilon(x):=\{y\in X\colon d(x,y)<\epsilon\},\quad
d_n(x,y):=\sup_{0\leqslant k\leqslant n-1} d(T^kx, T^ky),$$
and
$$B^n_\epsilon(x):=\{y\in X\colon d_n(x,y)<\epsilon\},$$
$$B_{\epsilon}^{n}(A):=\{y\in X: d(T^iy, T^i A)<\epsilon, 0\leqslant i\leqslant n-1\} \text{ for } A\subset X.$$
A set $E\subset X$ is
\emph{$(n,\epsilon)$-separated} if $d_n(x,y)>\epsilon$ for all $x\neq y \in E$ and
\emph{$(n, \epsilon)$-spanning for $Y$} if $Y\subset \bigcup_{x\in E}B^n_\epsilon(x)$.

The map $T$ is said to be
\emph{$\epsilon^*$-positively expansive}, if for any $x\ne y\in X$ there exists $n\in\mathbb N$ such that
$$ d(T^nx, T^ny)>\epsilon^*.$$
If $T$ is invertible, one considers expansiveness rather than positive expansiveness with $\mathbb N$ replaced by $\mathbb Z$ and the results in this paper can be proved analogously.
Suppose throughout the paper that  $P_n$ is finite for every $n$. This is true, for instance, for positively expansive maps.

For any finite subset $Z\subset X$, denote its cardinality by $|Z|$ and the uniform probability on $Z$ by $\nu_Z$, i.e.
$$\nu_Z(W):= \frac {|W\cap Z|}{|Z|} = \frac 1{|Z|}\sum_{z\in Z} \mathbbm 1_W(z)\qquad W\subset X.$$ For a real valued function $h$ on $X$, denote by $\mathbb E_Z(h)$ its expectation with respect to $\nu_Z$ and by $\sigma_Z^2(h)$ the variance when they exist.
Recall that the \emph{Birkhoff sums} are given by, for $n,m\in\mathbb N$,
$$ S_m^n h(x):=\sum_{i=m}^{m+n-1} h(T^i x),\quad\mbox{ and }\quad S^n h(x):= S_0^n h (x).$$
The \emph{oscillations} of $h: Y\to\mathbb R$, $Y$ being $P$ or $X$, around $x\in Y$ are given by
$$ \omega_m^n(h,\epsilon, x) :=\sup\Big\{ |S_m^n h (x)- S^n h(y)|\colon y\in B_\epsilon^{n}(T^m x)\cap Y\Big\}$$
and $\omega^n(h,\epsilon, x):=\omega_0^n(h,\epsilon, x)$.

The notation $f\lesssim g$ means that $f\leqslant Cg$ for some constant $C$. Denote the distribution function of the standard normal distribution by $\mathcal N(t)$.

\subsection{Specification}\label{sec:spec}

Specification property was introduced by Bowen in \cite{Bow71} (see also \cite[(21.1)]{DenGriSig76}), from which were derived many related definitions in the literature (see e.g. \cite{KwiLacOpr16}) \footnote{We thank W. Cordeiro for helpful discussions on this subject}. Each definition may lead to a CLT similar to the one proven here.
The notions of specification in this paper are defined as follows:

\begin{definition}\label{def:spec}
[Global Specification]
The dynamical system $(X,T)$ has the (global) specification property
 if for every $\epsilon>0$ there exists $M(\epsilon)\in \mathbb N$ such that: for any $M\geqslant M(\epsilon)$,
$x_1,...,x_k\in X$, $k\in \mathbb{N}$ and $n\in \mathbb{N}$   there exists a periodic point $p\in P_{k(n+M)}$ with
$$T^{(i-1)(n+M)}p\in B^n_\epsilon(x_i)\qquad i=1,...,k.$$
\end{definition}

\begin{definition}\label{def:specloc}
[Local Specification]
The dynamical system $(X,T)$ has the local specification property
if  for any $\epsilon>0$ there exist $\delta=\delta(\epsilon)>0$ and $N(\epsilon)\in \mathbb N$ such that for any $x_1,...,x_k\in X$, $k\in\mathbb N$ and  $n\geqslant N(\epsilon)$ with
$$ d(T^{n}x_i, x_{i+1}) <\delta\qquad i=1,...,k\mbox{ and }x_{k+1}=x_1$$
there exists a periodic point $p\in P_{kn}$ with
$$T^{(i-1)n}p\in B^n_\epsilon(x_i)\qquad i=1,...,k.$$
\end{definition}

For topologically mixing maps the local specification property implies the global specification (see e.g. \cite{Hasselblatt2017a}), often simply referred to as \emph{specification}. However, in absence of topological mixing, both notions are distinct.
Indeed, global specification implies that the map $T$ is topologically mixing. This does not need to be the case for maps with the local specification. On the other hand, Example 3 in Section \ref{sec:examples} shows that global specification does not imply local specification.

In order to simplify calculations and notations, our definition of global specification requires equal length for all the stretches, still it is possible to connect any two stretches of different lengths.
\begin{proposition}\label{prop:weakspec}
Suppose $(X,T)$ satisfies the global specification property. Then for every $\epsilon>0$, any $n_1, n_2\in\mathbb N$, $M_1, M_2\geqslant \tilde M(\epsilon)=M(\epsilon/3)$ and $x_1, x_2\in X$, there exists a periodic point $p\in P_{n_1+M_1+n_2+M_2}$ such that
$$p\in B_{\epsilon}^{n_1}(x_1)\text{~ and ~} T^{(n_1+M_1)}p\in B_\epsilon^{n_2}(x_2)\qquad i=1,2.$$
\end{proposition}

\begin{proof}
First we note that it suffices to show the statement for any $M_1=M_2\geqslant \tilde M(\epsilon)$, since $n_1$ and $n_2$ can be adjusted so that $M_1=M_2$. Given any $n_1, n_2\in\mathbb N$, $M\geqslant \tilde M(\epsilon)$ and $x_1, x_2\in X$, suppose $n_1\leqslant n_2$ wlog. By the specification property, there exists $p_1\in P_{n_1+M}$ such that $$p_1\in B_{\epsilon/3}^{n_1}(x_1).$$
Choose $k\in\mathbb N$ such that $k(n_1+M)\geqslant n_2-n_1.$ Use the specification property again to obtain $p_2\in P_{2(n_2+M)}$ such that $$p_2\in B_{\epsilon/3}^{n_2}(T^{k(n_1+M)-n_2+n_1}p_1)\text{~ and ~} T^{n_2+M}p_2\in B_{\epsilon/3}^{n_2}(x_2).$$ In particular,  noting that $T^{k(n_1+M)}p_1=p_1$, the first inclusion implies
$$T^{n_2-n_1}p_2\in  B_{\epsilon/3}^{n_2}(p_1)\subset B_{\epsilon/3}^{n_1}(p_1).$$
Now the specification property implies the existence of some $p\in P_{n_1+n_2+2M}$ with $$p\in B_{\epsilon/3}^{n_1+n_2+M}(T^{n_2-n_1}p_2),$$
in particular, $$p\in B_{\epsilon/3}^{n_1}(T^{n_2-n_1}p_2)\text{~ and ~} T^{n_1+M}p\in B_{\epsilon/3}^{n_2}(T^{n_2+M}p_2).$$
Therefore $p\in B_\epsilon^{n_1}(x_1)$ and $T^{n_1+M}p\in B_\epsilon^{n_2}(x_2).$
\end{proof}

\subsection{Examples of maps with specification} \label{sec:examples}
We now provide a few examples for which
these specification properties hold.

\subsubsection*{Example 1}

An \emph{Axiom A$^*$ homeomorphism} (c.f. \cite[pg 378]{Bow71}) is a homeomorphism $T:X\to X$ of a compact metric space $X$ which satisfies the following properties:
\begin{enumerate}
\item [(A1)]Periodic points are dense in $X$.
\item [(A2)]
For every $\epsilon>0$ there exists $\delta>0$ such that
$$d(x,y)<\delta\quad\Longrightarrow\quad W_\epsilon^u(x)\cap W_\epsilon^s(y)\ne \emptyset$$
where
$$ W_\epsilon^s(x):=\{y\in X: d(T^n(x), T^n(y))<\epsilon\quad \forall n\geqslant 0\}$$
and
$$ W_\epsilon^u(x):=\{y\in X: d(T^{-n}(x), T^{-n}(y))<\epsilon \quad\forall n\geqslant 0\}.$$
\item [(A3)]
There exist $\eta>0$, $c\geqslant 1$  and $0<\lambda<1$ such that for all $x\in X$ and $n\geqslant 0$
$$ d(T^n x,T^n y)< c\lambda^n d(x,y)\qquad y\in W_{\eta}^s(x)$$
and
$$ d(T^{-n}x,T^{-n}y)< c\lambda^n d(x,y)\qquad y\in W_{\eta}^u(x).$$
\end{enumerate}

Axiom A$^*$ homeomorphisms satisfy the local specification property. Topologically mixing Axiom A$^*$ homeomorphisms satisfy both the local and the global specification properties.

\subsubsection*{Example 2}
Consider a  dynamical system $(X,T)$ which possesses a Markov partition $\alpha$ of  sets $A$ which are contained in the closure of their interior. Then $(X,T)$ is a continuous factor of  a subshift of finite type $(Z,S)$. The periodic points in $Z$ of period $n$ correspond to sets in the refinement $\alpha_0^{n-1}$ and thus define measures on the $\sigma$-field generated by $\alpha_0^{n-1}$. Thus statements about measures on periodic points in $(Z,S)$ correspond to statements on natural measures on $(X,T)$ equipped with finite $\sigma$-fields. The results presented in this article can therefore be applied to this type of dynamical systems.

As an example, consider a hyperbolic rational map $T$ of the Riemannian sphere $\mathbb{S}^2$.
Ma\~n\'e~\cite{Man85} proved that $T$ is semi-conjugated to a
Bernoulli shift 
by a finite-to-one factor map  which is one-to-one almost everywhere.
Hence, the map $T$ on the completed Julia set equipped with the sequence topology trivially satisfies the local specification property.

\subsubsection*{Example 3}
A simple example of a system satisfying the global specification property which is not an Axiom A$^*$ homeomorphism is given by the subshift $\Sigma\subset \{1,...,N\}^{\mathbb Z}$ obtained by excluding a set of blocks from the full shift on $\{1,...,N\}^{\mathbb Z}$ as follows (see \cite[(17.1)]{DenGriSig76}).
For any $n\in\mathbb N$ pick two sequences $[p(n)]\neq [q(n)]$ of length $n$ from $\{1,..., N-1\}^{n}$. Exclude from the set of admissible words the cylinders of the form $\{[p(n)Nq(n)N]:n\in \mathbb N\}$. Consider the usual distance defined by $d(x,y):= r^k$ where $k=\min\{ |l|: x_{l}\ne y_{l} \ \mbox{\rm or}\ x_{-l}\ne y_{-l}\}$ for some $0<r<1$ and denote the shift transformation by $\sigma$. The subshift $\Sigma$ does not have the local specification property. To see this, use the blocks $[p(n)Nq(n)]$ and $[q(n)Np(n)]$ to generate periodic points $x(n)$ and $y(n)$ of period $2n+1$. Then for any $\delta>0$ and $n>-\log_{1/r}\delta$,
$$d(\sigma^{n+1}x(n),y(n))=r^n<\delta\quad\mbox{ and }\quad d(\sigma^{n+1}y(n),x(n))=r^n<\delta.$$
But if $p\in P_{4n+2}$ satisfies $d(\sigma^j p, \sigma^j x(n))<r/2$ and $d(\sigma^{n+1+j} p, \sigma^j y(n))<r/2$ for all $j=0,...,n$, then $p$ contains the excluded block $[p(n)Nq(n)N]$, a contradiction. It can be shown that such systems have the global specification property.

\section{A CLT under global specification}\label{sec:globalCLT}

The global specification property allows us to single out sets of periodic points which exhibit independent structure in the measure theoretical sense, and are spread out over the whole phase space.

\begin{definition}[$\epsilon$-Independence]\label{def:independence}
Let $\epsilon>0$ and $k,n, M\in\mathbb N$. A set ${\mathcal P}\subset P_{k(n+M)}$ is
$\epsilon$-independent if there exist a subset $E\subset P_{n+M}$
which $(n, 3\epsilon)$-spans $P_{k(n+M)}$ and a bijection $\Phi: E^k\to {\mathcal P}$ such that for any ${\bf x}=(x_1,...,x_{k})\in E^k$ and $1\leqslant i\leqslant k$,
\begin{equation}\label{eq:epsind}T^{(i-1)(n+M)}(\Phi({\bf x}))\in B^n_{\epsilon}(x_{i}).\end{equation}
The set $F:=E^k$ will be called the product set of $\mathcal P$.
\end{definition}

Recall that in the definition of global specification (Definition \ref{def:spec}) any given $\epsilon>0$ defines a $M(\epsilon)\in\mathbb N$.
\begin{proposition}\label{prop:indstr}
Let $(X,T)$ satisfy the (global) specification property.
 Then for any $\epsilon>0$, $k,n\in \mathbb N$ and $M\geqslant M(\epsilon)$ there exists a $\epsilon$-independent set $\mathcal P\subset P_{k(n+M)}$.
\end{proposition}
\begin{proof}
Choose a maximal $(n,2\epsilon)$-separated set $E\subset P_{n+M}$, then $E$ $(n,2\epsilon)$-spans $P_{n+M}$ by maximality.
For every $p\in P_{k(n+M)}$ there exists $x\in P_{n+M}\cap B^n_\epsilon(p)$ by specification.
Hence $E$ also $(n,3\epsilon)$-spans $P_{k(n+M)}$.
Again due to specification, for every ${\bf x}=(x_1,\ldots, x_k)\in E^k$, there exists $p\in P_{k(n+M)}$ such that $T^{(i-1)(n+M)}(p)\in B^n_{\epsilon}(x_i)$ for all $i=1,...,k$. This correspondence defines a map
$\Phi:E^k\mapsto P_{k(n+M)}$.
$\Phi$ is injective since $\Phi({\bf x})=\Phi({\bf y})$ implies $x_i\in B^n_{2\epsilon}(y_i)$ for all $i=1,...,k$. As $E$ is $(n,2\epsilon)$-separated, this implies $x_i=y_i$.
Hence ${\mathcal P}:=\Phi(E^k)$ is $\epsilon$-independent.
\end{proof}

We can use this structure to prove a Lindeberg type CLT for Birkhoff sums with respect to the uniform measure on $\epsilon$-independent sets. To this end, we approximate the measure on the $\epsilon$-independent sets by the uniform measure on the product sets.
Recall that for any finite $Z\subset X$, $h\in L^2(\nu_Z)$ and any $\eta> 0$ the \emph{Lindeberg function} is given by
$$ L_Z(h, \eta):= \int \big(h(z)-\mathbb E_Z(h)\big)^2 \mathbbm 1_{\{|h(z)-\mathbb E_Z(h)| >\eta\}}(z)d\nu_Z,$$
and that $\sigma_Z^2$ stands for the variance with respect to the uniform measure $\nu_Z$.
\begin{proposition}\label{prop:indclt} For every $l\in\mathbb N$ consider a finite subset $E_l\subset X$ and the Cartesian product set $F_l=E_l^{k_l}$ for some $k_l\in\mathbb N$. For each $1\leqslant i\leqslant k_l$ let $G_{l,i}:F_l \to\mathbb R$ be a function which depends only on the $i$-th component:
$$G_{l,i}(x_1,\ldots,x_{k_l}) =G_{l,i}(x_i),\quad \forall (x_1,\ldots,x_{k_l})\in F_l.$$
Set $\hat{s}^2_{l}=\sum_{i=1}^{k_l}  \sigma_{F_l}^2(G_{l,i}).$
Then the Lindeberg condition holds, i.e.
\begin{equation}\label{eq:prop1}
\lim_{l\to \infty}\frac 1 {\hat{s}_{l}^2}\sum_{i=1}^{k_l} L_{F_l}(G_{l,i}, \eta\hat s_l) =0, \quad \forall \eta>0
\end{equation}
if and only if the array is asymptotically negligible, i.e.
\begin{equation}\label{eq:propngl}
\lim_{l\to\infty}\max_{1\leqslant i\leqslant k_l}\nu_{F_l}(\{|G_{l,i}-\mathbb E_{F_l}(G_{l,i})|\geqslant \eta\hat s_l\})=0, \quad \forall\eta>0
\end{equation}
and the CLT holds, i.e.
\begin{equation}\label{eq:propclt}
\lim_{l\to\infty}\nu_{F_l}
\left(\bigg\{
\sum_{i=1}^{k_l}\left( G_{l,i} - \mathbb E_{F_l}(G_{l,i})\right)\leqslant t \hat{s}_{l}
\bigg\}\right)
=\mathcal{N}(t), \quad \forall t\in\mathbb R.
\end{equation}
\end{proposition}

\begin{proof}
Recall Lindeberg's CLT for independent random variables (see e.g. \cite[Theorem 15.43]{Klenke2014}): for an independent array of random variables, the Lindeberg condition holds if and only if the array is asymptotically negligible and the CLT holds. Then the proposition follows immediately since the functions $G_{l,i}$ form an independent array on $(F_l, \nu_{F_l})$.
\end{proof}

To realize the previously mentioned approximation, we require the observables to have controlled oscillations along the $\epsilon$-independent sets (see~\eqref{eq:4.2-0}) and to have negligible variance over the gaps (see~\eqref{eq:4.2-1}).
\begin{theorem}\label{theo:1}
Let $\{\epsilon_l>0, M_l, k_l, n_l\in \mathbb N\}_{l\in\mathbb N}$ be sequences of numbers, with $k_l\xrightarrow{l\to\infty}\infty$. Consider a sequence of $\epsilon_l$-independent sets $\mathcal P_{l}\subset P_{k_l(n_l+M_l)}$
and observables $h_l:P\to \mathbb R$ satisfying that
\begin{equation}\label{eq:4.2-0}
\lim_{l\to\infty}\,\frac{1}{s_l^j} \sum_{i=1}^{k_l}
\int \left(\omega_{a_i}^{n_l}(h_l,4\epsilon_l,p)\right)^j d\nu_{\mathcal P_{l}}(p)=0,\quad  j=1,2,
\end{equation}
where $a_{i}=(i-1)(n_l+M_l)$ and
$s_{l}^2=\sum_{i=1}^{k} \sigma^2_{\mathcal P_{l}}(S_{a_{i}}^{n_l} h_l)$, and
that
\begin{equation}\label{eq:4.2-1}
\lim_{l\to\infty}\, \frac1{s^2_l}\ \sigma^2_{\mathcal P_{l}}\left(\sum_{i=1}^{k_l} S_{a_{i}+n_l}^{M_l} h_l\right)=0.
\end{equation}
Then the Lindeberg condition
\begin{equation}\label{eq:epslind}
\lim_{l\to\infty}\frac 1{s_l^2} \sum_{i=1}^{k_l}L_{\mathcal P_l}(S_{a_i}^{n_l} h_l, \eta s_l) =0, \quad \forall \eta>0
\end{equation}
implies the CLT: for every $t\in\mathbb R$
\begin{equation}\label{eq:epsclt}
\lim_{l\to\infty}\nu_{\mathcal P_{l}}
\left(\Big\{
S^{k_l(n_l+M_l)} h_l-\mathbb{E}_{\mathcal{P}_{l}}(S^{k_l(n_l+M_l)} h_l)
\leqslant t s_l
\Big\}\right)=\mathcal{N}(t).\end{equation}

If
\begin{equation}\label{eq:4.2-00}
\lim_{l\to\infty}\,\frac{1}{\sigma^2_{\mathcal P_{l}}(S^{n_l} h_l)}
\int \left(\omega^{n_l}(h_l,2\epsilon_l,p)\right)^2d\nu_{\mathcal P_{l}}(p)=0
\end{equation}
holds additionally then the Lindeberg condition is also necessary.
\end{theorem}
\begin{proof}
First notice that condition~\eqref{eq:4.2-1} on the gaps implies that we can replace $S^{k_l(n_l+M_l)}h_l$ by the dynamical array $\sum_{i=1}^{k_l} S_{a_{i}}^{n_l} h_l$ without affecting the limit distribution.

Define an array $\{G_{l,i}\}$ on $F_l$, the product set of $\mathcal P_l$, by setting $$G_{l,i}(x_1,\ldots, x_{k_l}):=S^{n_l}h_l(x_i)$$ and let $\hat{s}^2_l=\sum_{i=1}^{k_l}\sigma^2_{F_l}(G_{l,i})$.
We show that $\frac{1}{s_l}\sum_{i=1}^{k_l}S_{a_i}^{n_l}h_{l}$ on $(\mathcal P_l, \nu_{\mathcal P_l})$ has the same limit distribution as $\frac{1}{\hat s_l}\sum_{i=1}^{k_l}G_{l,i}$ on $(F_l, \nu_{F_l})$, hence \eqref{eq:epsclt} is equivalent to \eqref{eq:propclt}.
In fact because \begin{equation}\label{eq:epsindaprox}
|G_i(\Phi^{-1}p)-S^{n_l}_{a_i}h_l(p)|\leqslant \omega_{a_i}^{n_l}(h_l,\epsilon_l, p),
\end{equation}
we have
for any $t>0$
$$
\nu_{\mathcal P_l}\left(\Big\{|\sum_{i=1}^{k_l}(G_{l,i}\circ \Phi^{-1}-S_{a_i}^{n_l} h_l|>ts_l\Big\}\right)\leqslant \frac{1}{t s_l}\sum_{i=1}^{k_l}\int \omega_{a_i}^{n_l}(h_l,\epsilon_l,p)d\nu_{\mathcal P_l}.
$$
Then as $l$ tends to $\infty$ the upper bound tends to $0$ by \eqref{eq:4.2-0}. Similar calculation shows that $$\lim_{\l\to\infty}{\hat{s}_l}/{s_l}=1.$$

Next we show that \eqref{eq:epslind} is equivalent to \eqref{eq:prop1}, consequently the first part of the theorem follows from Proposition \ref{prop:indclt}.
For any $\eta>0$, if $|G_{l,i}(\Phi^{-1}p)-\mathbb E_{F_l}(G_{l,i})|>\eta \hat{s}_l$ then by \eqref{eq:epsindaprox} either
$$|S_{a_i}^{n_l} h_l(p)-\mathbb E_{\mathcal P_l}(S_{a_i}^{n_l}h_l)|>\frac{\eta}3 s_l
$$
or else with
$$
\omega_{a_i}^{n_l}(h_l,\epsilon_l,p)>\frac{\eta}3 s_l
$$
since when $l$ is large enough \eqref{eq:4.2-0} yields
$$\int \omega_{a_i}^{n_l}(h_l,\epsilon_l,p)d\nu_{\mathcal P_l}\leqslant \eta \hat{s}_l-\frac{2\eta}{3} s_l.$$
This implies the equivalence of the Lindeberg conditions because
$$
  L_{F_l}(G_{l,i},\eta \hat{s}_l)\lesssim L_{\mathcal P_l}(S^{n_l}_{a_i} h_l, \frac\eta 3 s_l)  +  \int (\omega_{a_i}^{n_l}(h_l,\epsilon_l, p))^2 d\nu_{\mathcal P_l}
 $$
and \eqref{eq:4.2-0} and that the roles of $G_{l,i}$ and $S_{a_i}^{n_l}h_l$ can be switched.

Conversely we need to verify \eqref{eq:propngl}. We will show that the additional oscillation condition \eqref{eq:4.2-00} implies
\begin{equation}\label{eq:4.2-3}
\lim_{l\to\infty}\frac{\sigma^2_{\mathcal P_{l}}(S^{n_l}_{a_{i}}h_l)}{\sigma^2_{\mathcal P_{l}}(S^{n_l} h_l)} =1, \quad \text{ uniformly in } 1\leqslant i\leqslant k_l.
\end{equation}
Recall that there exists a bijection $\Phi$ between $\epsilon$-independent set $\mathcal P_l$ and its product set $F_l:=E_l^{k_l}$. Let $p\in \mathcal P_{l}$ and $\Phi^{-1}(p)=(x_1,\ldots, x_{k_l})\in F_l.$
Fix $i$, let $q_i=q_i(p):=\Phi(x_i, x_{i+1},...,x_{k_l}, x_1,...,x_{i-1})\in\mathcal P_{l}$ where $x_j:=x_{j \!\!\!\mod k_l}$ for $j\geqslant k_l$. Then for all $1\leqslant j\leqslant k_l$
$$T^{a_{j}}q_i\in B_\epsilon^{n_l}(x_{i+j-1}), \text{ hence } T^{a_{j}}q_i\in B_{2\epsilon}^{n_l}(T^{a_{i+j-1}}p).$$
Since $\Phi$ is bijective, so is the map $p\mapsto q_i(p)$ on $\mathcal P_{l}$. Therefore
\begin{align*}
&\quad |\sigma^2_{\mathcal P_{l}}(S_{a_{i}}^{n_l} h_l)-\sigma^2_{\mathcal P_l}(S^{n_l}h_l)|\notag\\
&=\frac{1}{|\mathcal P_{l}|} \cdot\left|\sum_{p\in \mathcal P_{l}}\Big(S^{n_l}_{a_i}h_l(p)-\mathbb E_{\mathcal P_{l}}(S_{a_i}^{n_l}h_l)\Big)^2-\sum_{p\in \mathcal P_{l}}\Big(S^{n_l}h_l(q_i)-\mathbb E_{\mathcal P_{l}}(S^{n_l}h_l)\Big)^2\right|\\
&\leqslant \frac{1}{|\mathcal P_{l}|} \cdot\sum_{p\in\mathcal P_{l}}\bigg(\Delta_p^2+2\Delta_p\cdot \big|S^{n_l}h_l(q_i)-\mathbb E_{\mathcal P_{l}}(S^{n_l}h_l)\big|\bigg)\notag\\
&\leqslant \int \Delta_p^2 \,d\nu_{\mathcal P_{l}}+2\Big(\int \Delta_p^2\,d\nu_{\mathcal P_{l}}\Big)^{1/2}\cdot\sigma_{\mathcal P_{l}}(S^{n_l}h_l)
\end{align*}
where
\begin{align*}
\Delta_p&:=\big|(S^{n_l}_{a_{i}}h(p)- S^{n_l}h_l(q_i))-(\mathbb E_{\mathcal P_{l}}(S_{a_i}^{n_l} h_l)-\mathbb E_{\mathcal P_{l}}(S^{n_l}h_l))\big|\\
&\leqslant \omega^{n_l}(h_l, 2\epsilon_l, q_i)+ \int \omega^{n_l}(h_l, 2\epsilon_l, p)\, d\nu_{\mathcal P_{l}}(p).
\end{align*}
It follows that
$$\left|\frac{\sigma^2_{\mathcal P_{l}}(S^{n_l}_{a_i}h_l)}{\sigma^2_{\mathcal P_{l}}(S^{n_l}h_l)}-1\right|\lesssim \frac{\int (\omega^{n_l}(h_l,2\epsilon_l, p))^2 \nu_{\mathcal P_{l}}}{\sigma^2_{\mathcal P_{l}}(S^{n_l}h_l)}+\frac{(\int (\omega^{n_l}(h_l,2\epsilon_l, p))^2 \nu_{\mathcal P_{l}})^{1/2}}{\sigma_{\mathcal P_{l}}(S^{n_l}h_l)}$$
and hence \eqref{eq:4.2-3}.
Therefore $\{\frac{1}{\hat s_l}(G_{l,i}-\mathbb E_{F_l}(G_{l,i}))\}$ is asymptotically negligible, since
$$\nu_{F_l}(\{|G_{l,i}-\mathbb E_{F_l}(G_{l,i})|\geqslant \eta\hat s_l\})\leqslant \frac{\sigma^2_{F_l}(G_{l,i})}{\eta^2\hat s_l^2}\lesssim \frac{\sigma^2_{\mathcal P_l}(S_{a_i}^{n_l}h_l)}{\eta^2 s_l^2}\lesssim \frac{1}{\eta^2k_l}\xrightarrow{l\to\infty} 0.$$
\end{proof}

Note that a $\epsilon$-independent set $\mathcal P\subset P_{k(n+M)}$ is generally a proper subset of $P_{k(n+M)}$. Nevertheless we still can obtain much information in the limit of the uniform measures on $\mathcal P$ as if in the case of $P_{k(n+M)}$ (see Theorem \ref{theo:weak-convergence}).
On the other hand, we consider here a weighted measure with support $P_{k(n+M)}$ coming naturally from $\mathcal P$ and show that a Lindeberg CLT for this weighted measure can be deduced from the CLT for the uniform measure.
 For every $p\in \mathcal P$ and its counterpart in the product set $(x_1,\ldots,x_k)=\Phi^{-1}(p)\in E^k$, let
$$Q(p):=\{q\in P_{k(n+M)}: d_n(T^{(i-1)(n+M)}q, x_i)<3\epsilon, \forall 1\leqslant i\leqslant k\}.$$
Because $E$ is $(n,3\epsilon)$-spanning for $P_{k(n+M)}$ it follows that
$$P_{k(n+M)}\subset\bigcup_{p\in\mathcal P_{\epsilon}}Q(p)$$
and for every $p\in \mathcal P$ and $q \in Q(p)$, one has
\begin{equation}\label{eq:qpeps}
 d_n(T^{(i-1)(n+M)} p, T^{(i-1)(n+M)} q )< 4\epsilon, \quad \forall 1\leqslant i\leqslant k.
 \end{equation}
However, a point $q\in P_{k(n+M)}$ may belong to multiple $Q(p)$. To account for this multiplicity define a weighted probability measure
on $P_{k(n+M)}$ by
$$\nu^w(A):=\sum_{q\in P_{k(n+M)}}w(q)\,\mathbbm 1_A(q)$$
where
$$w(q):=\frac1{|\mathcal P|}\sum_{\{p\colon q\in Q(p)\}}\frac1{|Q(p)|}.$$

\begin{corollary}\label{cor:1}
Consider $\epsilon_l$-independent sets $\mathcal P_l\subset P_{k_l(n_l+M_l)}$ and observables $h_l$ satisfying the oscillation condition \eqref{eq:4.2-0} and that
$$\lim_{l\to\infty}\frac{k_l^2M_l^2\|h_l\|^2_\infty}{s_l^2}=0.$$
Then the Lindeberg condition \eqref{eq:epslind} implies the CLT with respect to the weighted measure $\nu_l^w$
$$\lim_{l\to\infty}\nu^w_l\left(\Big\{
S^{k_l(n_l+M_l)}h_l
- \mathbb E_{\nu^w_l} (S^{k_l(n_l+M_l)}h_l) \leqslant t s_l\Big\}\right)=\mathcal{N}(t).$$

\end{corollary}
\begin{proof}
First note that \eqref{eq:4.2-1} is satisfied, because
$$\frac{\sigma^2_{\mathcal P_{l}}(\sum_{i=1}^{k_l} S_{a_{i}+n_l}^{M_l} h_l)}{s^2_l} \lesssim \frac{k_l^2M_l^2\|h_l\|_\infty^2}{s_l^2}\xrightarrow{l\to\infty} 0$$
by the assumption, as well as that $$\lim_{l\to\infty}\frac1{s^2_l}\sigma^2_{\nu_l^w}\left(\sum_{i=1}^{k_l} S_{a_{i}+n_l}^{M_l} h_l\right)=0.$$
So it suffices to show that the dynamical array $\{\frac{1}{s_l}\sum_{i=1}^{k_l}S_{a_i}^{n_l} h_l\}$ has the same limit distribution with respect to $\nu_{\mathcal P_l}$ and to $\nu_l^w$ when centered accordingly.
Note that for every pair of $p\in \mathcal P_{l}$ and $q \in Q_l(p)$,  it follows from \eqref{eq:qpeps} that
$$
\left| \sum_{i=1}^{k_l} \left(S^{n_l}_{a_i} h_l(p) -  S^{n_l}_{a_i} h_l(q)\right) \right|\leqslant \sum_{i=1}^{k_l}  \omega_{a_i}^{n_l}(h_l,4\epsilon_l, p)=:\Omega_l(p).
$$
 Denote
\begin{eqnarray*}
&&U_l(t):=\left\{q\in P_{k(n+M)}\colon \sum_{i=1}^{k_l} \left(S_{a_i}^{n_l} h_l (q)-\mathbb{E}_{\nu^w_{l}}(S^{n_l}_{a_i} h_l)\right)\leqslant t s_{l}\right\}\\
&&V^+_l(t) :=\left\{p\in \mathcal P_{l}\colon \sum_{i=1}^{k_l} \left(S^{n_l}_{a_i} h_l(p)-\mathbb E_{\mathcal P_{l}}(S^{n_l}_{a_i} h_l)\right)\leqslant t s_{l} + \Omega_l(p)+\mathbb{E}_{\mathcal P_{l}}(\Omega_l)\right\}\\
&&V^-_l(t) :=\left\{p\in \mathcal P_{l}\colon \sum_{i=1}^{k_l} \left(S^{n_l}_{a_i} h_l(p)-\mathbb E_{\mathcal P_{l}}(S^{n_l}_{a_i} h_l)\right)\leqslant t s_{l} - \Omega_l(p)-\mathbb{E}_{\mathcal P_l}(\Omega_l)\right\}.
\end{eqnarray*}
Then
$$\bigcup_{q\in U_l(t)}\{p\in\mathcal P_{l}: q\in Q_l(p)\}\subset V^+_l(t),
\quad \bigcup_{p\in V^-_{l}(t)} Q_l(p)\subset U_l(t)$$
and thus
\begin{align*}
\nu^w_{l}( U_l(t))&=\frac 1{|\mathcal P_l|} \sum_{q\in U_l(t)}\sum_{\{p\colon q\in
Q_l(p)\}}\frac{1}{|Q_l(p)|}\\
&\leqslant \frac{1}{|\mathcal P_{l}|}{\sum_{p\in V^+_l(t)}}\sum_{q\in Q_l(p)}\frac{1}{|Q_l(p)|}=\nu_{\mathcal P_{l}}(V^+_l(t)).
\end{align*}
Similarly $\nu^w_{l}( U_l(t))\geqslant \nu_{\mathcal P_{l}}(V_l^-(t)).$
Observe that because of \eqref{eq:4.2-0} $\frac 1{s_l}\Omega_l$ converges to $0$ in probability $\nu_{\mathcal P_l}$.
So $\displaystyle\lim_{l\to\infty}\nu_{\mathcal P_{l}}(V_l^+(t))=\lim_{l\to\infty}\nu_{\mathcal P_{l}}(V_l^-(t))=\mathcal N(t)$ by Theorem \ref{theo:1},  and hence $\displaystyle\lim_{l\to\infty}\nu^w_l( U_l(t))=\mathcal N(t)$ as desired.
\end{proof}

\section{A CLT under local specification}\label{sec:localCLT}
Compared to the global specification property, the local specification property allows us to single out sets of periodic points with independence structure in a local scenario in which the positions of certain orbits are specified a priori.

\begin{definition}[Local $\epsilon$-Independence]\label{def:locidp}
Let $\epsilon>0$ and $\mathcal U$  be a family of open sets. Let $A\in \bigvee_{i=0}^{k-1} T^{-in}\mathcal U$ for some $k, n\in\mathbb N$. A subset $\mathcal P\subset P$
is locally $\epsilon$-independent with respect to $A$
if there exist $E_i\subset T^{(i-1)n}A$, $1\leqslant i\leqslant k$,
and a bijection
$\Phi$ from $F:=\prod_{i=1}^{k} E_i$ to ${\mathcal P}$ such that for any ${\bf x}=(x_1,...,x_{k})\in F$ and  $1\leqslant i\leqslant k$.
\begin{equation}\label{eq:6.2}
T^{(i-1)n}(\Phi({\bf x}))\in B^n_{\epsilon}(x_{i}).
\end{equation}
The set $F$ will be called the product set of $\mathcal P$.
\end{definition}

Recall that in the definition of local specification (Definition \ref{def:specloc}) any given $\epsilon>0$ defines a $N(\epsilon)\in\mathbb N$ and a $\delta(\epsilon)>0$.
\begin{proposition}\label{prop:locidp}
Let $(X,T)$ satisfy the local specification property. Then for any $\epsilon>0$, $k\in \mathbb N$, $n\geqslant N(\epsilon)$ and any family $\mathcal U$ of open subsets of diameter at most $\delta(\epsilon)$ and $A\in\bigvee_{i=0}^{k-1} T^{-in}\mathcal U$,
there exists a locally $\epsilon$-independent set $\mathcal P\subset P_m$ with respect to $A$,
 where $m\geqslant kn$ is any given multiple of $n$, and such that
 if in addition the system is $\epsilon^*$-expansive and $\epsilon<\epsilon^*/3$ then $$A\cap P_m\subset \mathcal P\subset B_{\epsilon}^{m}(A)\cap P_m.$$
\end{proposition}

\begin{proof}
For each $1\leqslant i\leqslant k-1$ choose a maximal $(n,2\epsilon)$-separated set $E_i\subset T^{(i-1)}A$ and a maximal $(m-(k-1)n, 2\epsilon)$-separated set $E_k\subset T^{(k-1)n}A\cap T^{-m+(k-1)n}A$. For any $(x_1,\ldots,x_k)\in F=\prod_{i=1}^{k} E_i$, by the local specification property  there exists $p\in P_m$ such that  $$T^{(i-1)n}p\in B_{\epsilon}^n(x_i) \qquad 1\leqslant i\leqslant k-1$$ and $$T^{(k-1)n}p\in B_{\epsilon}^{m-(k-1)n}(x_k).$$
It defines a map $\Phi$ from $F$ to $P_m$. This map is injective as $E_i$ is $(n,2\epsilon)$-separated for each $1\leqslant i\leqslant k$, hence its image, denoted by $\mathcal P$, is a locally $\epsilon$-independent set with respect to $A$.
Clearly $\mathcal P\subset B_\epsilon^m(A)\cap P_m$.

Suppose now the system is $\epsilon^*$-expansive and $\epsilon<\epsilon^*/3$. For any $q\in A\cap P_m$, due to maximality of $E_i$, there exists $\mathbf y\in F$ such that $d_n(y_i, T^{(i-1)n}q)<2\epsilon, 1\leqslant i\leqslant k-1$ and $d_{m-(k-1)n}(y_k, T^{(k-1)n}q)<2\epsilon$. Therefore expansiveness implies that $q=\Phi(\mathbf y)\in\mathcal P$.
\end{proof}

The local $\epsilon$-independence structure also entails a Lindeberg type CLT. 

\begin{theorem}\label{theo:2}

Let $\{\epsilon_l>0, k_l, n_l\in\mathbb N\}_{l\in\mathbb N}$ be sequences of numbers. Let $\mathcal U_l$ be a family of open sets and $A_l\in \bigvee_{i=0}^{k_l-1}T^{-in_l}\mathcal U_l$. Consider a sequence of locally $\epsilon_l$-independent set $\mathcal P_l$ with respect to $A_l$
and observables $h_{l,i}:P\to\mathbb R$ satisfying  that
$$\lim_{l\to\infty}\frac 1{s_l^j} \sum_{i=1}^{k_l} \int \left(\omega_{a_i}^{n_l}(h_{l,i}, \epsilon_l, p)\right)^j d\nu_{\mathcal P_l}(p)=0, \quad j=1,2,$$
where $a_i=(i-1)n_l$ and
$s_l^2=\sum_{i=1}^{k_l} \sigma^2_{\mathcal P_l}(S^{n_l}_{a_i}h_{l,i})$.
Then the Lindeberg condition holds, i.e.
$$\lim_{l\to\infty}\frac 1{s_l^2} \sum_{i=1}^{k_l}
  L_{\mathcal P_l}(S^{n_l}_{a_i}h_{l,i}, \eta s_l)=0,\qquad\forall\eta>0$$
if and only if
the array $\big\{\frac{1}{s_l}\left(S^{n_l}_{a_i}h_{l,i}-\mathbb E_{\mathcal P_l}(S^{n_l}_{a_i}h_{l,i})\right)\big\}_{1\leqslant i\leqslant k_l, l\in\mathbb N}$ is $\nu_{\mathcal P_l}$-asymptotically negligible, i.e.
$$\lim_{l\to\infty}\max_{1\leqslant i\leqslant k_l}\nu_{\mathcal P_l}\left(\left\{\left|S^{n_l}_{a_i}h_{l,i}-\mathbb E_{\mathcal P_l}(S^{n_l}_{a_i}h_{l,i})\right|\geqslant \eta s_l\right\}\right)=0,\qquad\forall\eta>0$$
and the CLT holds, i.e.
for every $t\in \mathbb R$
$$ \lim_{l\to\infty} \nu_{\mathcal P_l} \left(\bigg\{ \sum_{i=1}^{k_l} \left(S^{n_l}_{a_i}h_{l,i}-\mathbb E_{\mathcal P_l}(S^{n_l}_{a_i}h_{l,i})\right)\leqslant ts_l\bigg\}\right) = \mathcal N(t).$$
\end{theorem}

\begin{proof}
Like the proof of Theorem \ref{theo:1}, define an array $\{G_{l,i}\}$ on $F_l$, the product set of $\mathcal P_l$, by setting
\begin{equation*}
G_{l,i}(x_1,...,x_{k_l}):= S^{n_l} h_{l,i}(x_i), \quad 1\leqslant i\leqslant k_l
\end{equation*}
and let $\hat{s}_l^2=\sum_{i=1}^{k_l} \sigma^2_{F_l}(G_{l,i})$.
Due to the structure of locally independent sets, specifically by \eqref{eq:6.2}, for any $p\in\mathcal P_l$ one has similar to \eqref{eq:epsindaprox}
$$\left|G_{l,i}(\Phi^{-1}p)-S^{n_l}_{a_i} h_{l,i}(p)\right|\leqslant \omega^{n_l}_{a_i}(h_{l,i}, \epsilon_l, p).$$
Then one can show that each statement for the dynamical array $\{\frac{1}{s_l}(S_{a_i}^{n_l}h_{l,i}-\mathbb E_{\mathcal P_l}(S^{n_l}_{a_i}h_{l,i}))\}$ in the theorem is equivalent to the statement for the array $\{\frac{1}{\hat s_l}(G_{l,i}-\mathbb E_{F_l}(G_{l,i}))\}$ in Proposition \ref{prop:indclt}, much in the same way as illustrated in Theorem \ref{theo:1}.
\end{proof}

\section{Applications: Measures of maximal entropy}\label{sec:mme}
In this section we provide application to system with a unique measure of maximal entropy, since such a measure is the weak-* limit of the equidistributed Dirac measures on periodic points. However, we believe our methods apply in more generality.

Note that in a positively expansive system satisfying global specification property the measure of maximal entropy is unique (by adapting Bowen's proof \cite{Bow71, Bow74} in the homeomorphism case or by taking the natural extension \cite{Rue92}).
Under certain conditions the measure of maximal entropy is in fact the limit of the equidistributed Dirac measures on the $\epsilon$-independent sets.
This is the case if, for instance:
\begin{equation}\label{eq:pphomeo}
T|_{P_n} \text{ is a homeomorphism for every } n. \tag{*}\footnote{This is automatic if $T$ is a homeomorphism.}
\end{equation}

Recall that in the global specification property the length of the gap $M(\epsilon)$ depends on $\epsilon>0$ (see Definition~\ref{def:spec} and Proposition~\ref{prop:weakspec}). In this section, for a sequence $\{\epsilon_l\}_{l\in\mathbb N}$, $M(\epsilon_l)$ will be abbreviated to $M_l$.
\begin{theorem}\label{theo:weak-convergence}
Let $(X,T)$ be a $\epsilon^*$-positively expansive system satisfying global specification property. Let $\mu$ be the unique measure of maximal entropy. Assume condition \eqref{eq:pphomeo}.
Given any $0<\epsilon<\epsilon^*/8$, then for any sequence of integers $k_l\in\mathbb N$, $n_l\xrightarrow{l\to\infty}\infty$ and
$\epsilon$-independent sets $\mathcal P_l\subset P_{k_l(n_l+M(\epsilon))}$, $$\nu_{\mathcal P_l}\xRightarrow{l\to\infty}\mu.$$
\end{theorem}

\begin{proof}
Note that a positively expansive transformation has finite entropy. We first prove that
any weak accumulation measure $\nu$ is $T$-invariant.
Let wlog $\nu=\text{w-}\lim_{l\to\infty}\nu_{\mathcal{P}_l}$. Recall $E_l\subset P_{k_l(n_l+M)}$  from  Definition~\ref{def:independence}
with the bijection $\Phi_l: E_l^{k_l}\to \mathcal P_l$. In the rest of the proof we will write $k,n$ for $k_l, n_l$ for simplicity, where the dependence is self-evident.

Given $p\in\mathcal P_l$, since $E_l$ $(n,3\epsilon)$-spans $P_{k(n+M)}$ and $Tp\in P_{k(n+M)}$, there exists $\mathbf{y}=(y_1,\ldots,y_{k})\in E_l^k$ such that for any $1\leqslant i\leqslant k$
$$ T^{(i-1)(n+M)}\left(Tp\right)\in B^n_{3\epsilon}(y_i).$$
Let $q(p):=\Phi_l(\mathbf{y})\in\mathcal P_l$, then
$T^{(i-1)(n+M)}(q(p))\in B_\epsilon^n(y_i)$,
hence
$$d_n(Tp, q(p))\leqslant d_n(Tp, y_1)+d_n(y_1, q(p))<4\epsilon<\epsilon^*.$$
Since $T$ is $\epsilon^*$-positively expansive and $n\xrightarrow{l\to\infty}\infty$,
$$ \lim_{l\to \infty} d(Tp, q(p))=0$$
uniformly in $p$, that is, for any $\delta>0$ there exists $\ell(\delta)\in\mathbb N$ such that for any $l>\ell(\delta)$ and any $p\in\mathcal P_l$, $d(Tp, q(p))<\delta$.
Moreover, the map $p\mapsto q(p)$ is injective:
for $p, p'\in \mathcal P_l\subset P_{k(n+M)}$, if $q(p)=q(p')$ then they correspond to the same $\mathbf y\in E_l^k$, hence for all $1\leqslant i\leqslant k$
$$d_n(T^{(i-1)(n+M)}Tp, T^{(i-1)(n+M)}Tp')< 6\epsilon<\epsilon^*.$$
Therefore uniformly
$$\lim_{l\to\infty} d(T^{(i-1)(n+M)}Tp, T^{(i-1)(n+M)}Tp')=0.$$
With condition \eqref{eq:pphomeo}
it implies that for large $l$ and all $1\leqslant i\leqslant k$
$$d_M(T^{i(n+M)+n}Tp, T^{i(n+M)+n}Tp')<\epsilon^*.$$
Thus orbits of $p$ and $p'$ are always $\epsilon^*$-close and $p=p'$ by expansiveness.
Finally, for any $\delta>0$, any Lipschitz function $f$ and $l>\ell(\delta)$,
$$\left| \int  f(Tp) \,d\nu_{\mathcal P_l}(p) - \int f(p) \,d\nu_{\mathcal P_l}(p)\right| \leqslant
\frac{1}{|\mathcal P_l|}\sum_{p\in \mathcal P_l}
|f(Tp)-f(q(p))| \leqslant \delta\cdot Lip(f).$$
Letting $l\to\infty$ proves the $T$-invariance of $\nu$.

We now prove that the metric entropy $h(\nu)$ of any weak accumulation point $\nu$ of the sequence $\nu_{\mathcal{P}_l}$ agrees with the topological entropy $h(T)$.

Let $\alpha$ be a partition of $X$ into Borel sets $A\in\alpha$ of diameter $<\epsilon$ with $\nu(\partial A)=0$, where $\partial A$ denotes the boundary of $A$.
For any $m\geqslant 1$ and $A\in \bigvee_{i=0}^{m-1} T^{-i}\alpha$ one has
$$
\nu(A)=\lim_{l\to\infty} \nu_{\mathcal P_l}(A)=\lim_{l\to\infty}
\frac{|\mathcal P_l\cap A|}{|\mathcal P_l|}.$$
Denote the maximal cardinality of all $(m, \epsilon^*)$-separated sets of $X$ by $r_{m,\epsilon^*}(X)$. We claim there exists $C=C(\epsilon)>0$ such that
\begin{equation}
\label{claim:upper-bound}
\lim_{l\to\infty}
\frac{|\mathcal P_l\cap A|}{|\mathcal P_l|}\leqslant
\frac{C}{r_{m,\epsilon^*}(X)}.
\end{equation}
Assuming this claim it follows that
$$ \frac 1m H_\nu\left(\bigvee_{i=0}^{m-1} T^{-i}\alpha\right) \geqslant \frac 1m \log \left(\frac{ r_{m,\epsilon^*}(X)}{C}\right)$$
and, letting $m\to\infty$ yields
$$ h_\nu(T,\alpha)\geqslant  h(T).$$
Since $(X,T)$ is $\epsilon^*$-positively expansive, partitions by sets of diameter $<\epsilon^*$ are generating and thus $h_\nu(T)=h_\nu(T,\alpha)\geqslant  h(T)$.
The reversed inequality is well-known.

We now prove claim~\eqref{claim:upper-bound}. Assume $l$ large such that $n>m+2M-1$ and consider $p\in\mathcal P_l\cap A$. Let $(x_1, \ldots,x_{k})=\Phi_l^{-1}(p)$, then $p\in B_\epsilon^n(x_1)$ and $\displaystyle{\sup_{0\leqslant i\leqslant m-1}diam(T^iA)<\epsilon}$, hence
$$
x_1\in \hat B^m_{2\epsilon}(A):=\{x:d_m(x,y)<2\epsilon, \forall y\in A\}.
$$
It follows that
$$|\mathcal P_l\cap A|\leqslant |E_l\cap \hat B^m_{2\epsilon}(A)|\cdot |E_l|^{k-1}.$$
Given a $(m,\epsilon^*)$-separated set $R$ of maximal cardinality, define a map
$$\kappa: R\times \{E_l\cap \hat B_{2\epsilon}^m(A)\} \to  E_l$$
for which the number of preimages is bounded by a constant $C(\epsilon)$ so that
$$ r_{m,\epsilon^*}(X)\cdot |E_l\cap \hat B^m_{2\epsilon}(A)| \leqslant C |E_l|,$$
which in turn implies claim~\eqref{claim:upper-bound}.
Define $\kappa$ as follows. By
Proposition~\ref{prop:weakspec} for any $(z,x)\in R\times \{E_l\cap \hat B_{2\epsilon}^m(A)\}$, and since $n>m+2M-1$, there exists $p\in P_{n+M}$ with
$$p\in B_{\epsilon}^m(z) \quad \text{ and  }\quad
 T^{m+M}p\in B_\epsilon^{n-(m+M)}\left(T^{m+M}x\right).$$
Since $E_l$ is $(n,3\epsilon)$-spanning for $P_{k(n+M)}$ there exists $y\in E_l\cap  B_{3\epsilon}^n(p)$.
The map $\kappa$ is defined by choosing $\kappa(z,x):=y$.
To prove the uniform bound on the number of pre-images of $\kappa$, assume
$\kappa(z,x)=\kappa(z',x')=y$ for some $y\in E_l$.
Denote by $p, p'\in P_{n+M}$ the associated periodic points. It follows that
$$ d_n(p,p')< 6\epsilon\quad\mbox{ and }\quad
d_m(z,z')< 8\epsilon<\epsilon^*.$$
But since $z, z'\in R$ and $R$ is $(m,\epsilon^*)$-separated, it follows that $z=z'$.
Moreover, since $x,x'\in \hat B_{2\epsilon}^m(A)$
we have
$$ d_m(x,x')< 4\epsilon<\epsilon^*$$
and also
$$ d_{n-(m+M)}(T^{m+M}x, T^{m+M}x')<8\epsilon<\epsilon^*.$$
Therefore if $x\ne x'\in E_l\subset P_{n+M}$, there exists $t\in [m,m+M]\cup[n,n+M]$ with
$$ d(T^t(x), T^t(x'))>\epsilon^*.$$
Let $c$ denote the minimal amount of balls of radius $<\epsilon^*/2$ necessary to cover $X$. Then for fixed $x\in E_l\cap \hat B_{2\epsilon}^m(A)$ there are at most $c^{2M}$ points $x'$ mapping to the same image under $\kappa$. Hence setting
$ C:= c^{2M}$ proves the statement.
\end{proof}

Moreover, applying a diagonal argument to the previous theorem yields:

\begin{corollary} \label{cor:wkcvg}
Let $(X,T)$ be a positively expansive system satisfying the global specification property. Let $\mu$ be the unique measure of maximal entropy. Assume condition \eqref{eq:pphomeo}. Then, given sequences $\{\epsilon_l>0, k_l\in\mathbb N\}_{l\in\mathbb N}$ with $\epsilon_l\to 0$ and $k_l\to\infty$, there exists $n_l^*\to\infty$ such that for any sequence $\{n_l\}$ with $n_l> n^*_l$ and any $\epsilon_l$-independent sets $\mathcal P_l\subset P_{k_l(n_l+M_l)}$, $\nu_{\mathcal P_l}\Rightarrow \mu$.
\end{corollary}

\begin{theorem}\label{theo:3}
Let $(X,T)$ be a positively expansive system satisfying the global specification property. Let $\mu$ be the unique measure of maximal entropy.  Assume condition \eqref{eq:pphomeo}. Let $\{\epsilon_l>0, k_l, n_l\in\mathbb N\}_{l\in\mathbb N}$ be sequences such that $\epsilon_l\to 0, k_l^{1-\eta_l^*}\to\infty, \epsilon_lk_l^{(1+\eta_l^*)/2}\to 0, n_l>n^*_l$ and ${k_l^{1+\eta_l^*} M_l^2}/{n_l^2}\to 0$ as $l\to\infty$ for some $\eta_l^*\to1^-$.
Consider  a sequence of $\epsilon_l$-independent sets ${\mathcal P}_{l}\subset P_{k_l(n_l+M_l)}$ and a Lipschitz function $h$.
Then
$$
\lim_{l\to\infty}\nu_{\mathcal P_{l}}
\left(\bigg\{
\frac{1}{k_l(n_l+M_l)}\big(S^{k_l(n_l+M_l)}h -\mathbb{E}_{\mathcal P_{l}}(S^{k_l(n_l+M_l)}h)\big)
\leqslant k_l^{-\frac12+\eta} \bigg\}\right)=1.
$$ for any $\eta>0$.

Moreover, if
$$\sum_{l\in\mathbb N} \frac{1}{\sqrt{k_l}}\ \cdot\, \sup_{\|h\|_{\rm Lip}\leqslant 1}\ \frac {\mathbb E_ {{\mathcal  P}_{l}} (S^{n_l}h)^4}{\sigma^4_{{\mathcal P}_{l}}(S^{n_l}h)} <\infty,$$
then one has
$$
\frac 1{k_l(n_l+M_l)} \sum_{i=0}^{k_l(n_l+M_l)-1}\delta_{T^i(p_{l})}
\xRightarrow{l\to\infty}\mu
$$
for $\nu$ a.e. sequence  $\displaystyle{\{p_{l}\}_{l\in\mathbb N}\in \prod_{l\in\mathbb N}{\mathcal  P}_{l}}$
where $\nu$ denotes the product measure of $\nu_{{\mathcal  P}_{l}}$ on $P^{\mathbb N}$.
\end{theorem}

\begin{proof}
Let  $a_{i}=(i-1)(n_l+M)$ and
$s_{l}^2=\sum_{i=1}^{k_l} \sigma^2_{\mathcal P_{l}}(S_{a_{i}}^{n_l} h)$
$$I:=\{ l\in\mathbb N:   s_l^2 \geqslant {k_l}^{1-\eta_l^*}  n_l^2\}.$$
Assume $I$ has infinitely many elements. We verify conditions in Theorem \ref{theo:1} for $l\in I$.
Since $h$ is Lipschitz, one has for $j=1,2$ that
$$\frac1{s_l^j}\sum_{i=1}^{k_l}
\int \left(\omega_{a_i}^{n_l}(h,4\epsilon_l,p)\right)^j d\nu_{\mathcal P_{l}}\lesssim \frac{k_l(\epsilon_l n_l)^j}{(k_l^{1-\eta_l^*} n_l^2)^{j/2}}=k_l^{1-j/2+\eta_l^*j/2}\epsilon_l^{j}\xrightarrow{l\to\infty} 0.$$
and
$$\frac1{s^2_l}\ \sigma^2_{\mathcal P_{l}}\left(\sum_{i=1}^{k_l} S_{a_{i}+n_l}^{M_l} h\right)\lesssim \frac{k_l^2M_l^2\|h\|_\infty^2}{k_l^{1-\eta_l^*} n_l^2}\xrightarrow{l\to\infty}0$$
by assumptions. Also since
$${|S_{a_i}^{n_l} h(p)-\mathbb E_{\mathcal P_l}(S_{a_i}^{n_l} h)|}/{s_l}\lesssim n_l/s_l\leqslant 1/k_l^{1/2-\eta_l^*/2}\xrightarrow{l\to\infty}0,$$
$$\lim_{l\to\infty}\frac 1{s_l^2} \sum_{i=1}^{k_l}L_{\mathcal P_l}(S_{a_i}^{n_l} h_l, \eta s_l)=0\quad  \forall \eta>0.$$
Hence Theorem \ref{theo:1} implies that for any subsequence $l\in I$ and any $\eta>0$
\begin{align}
&\quad \lim_{l\to\infty}\nu_{\mathcal P_l}\left(\bigg\{\frac{1}{k_l(n_l+M_l)}\big(S^{k_l(n_l+M_l)}h -\mathbb{E}_{\mathcal P_{l}}(S^{k_l(n_l+M_l)}h)\big)
\leqslant k_l^{-1/2+\eta}\bigg\}\right)\label{eq:4.2-5}\\
&=\lim_{l\to\infty} \mathcal N(k_l^{1/2+\eta}(n_l+M_l)/s_l)=1 \notag
\end{align}
since $s_l^2\lesssim k_l n_l^2$, and hence
$k_l^{1/2+\eta}(n_l+M_l)/s_l\gtrsim k_l^{1/2+\eta}(n_l+M_l)/(k_l^{1/2}n_l)\xrightarrow{l\to\infty}\infty.$

On the other hand, for any subsequence $l\not\in I$, where $\eta>(1-\eta^*_l)/2$
\begin{align*}
&\quad \nu_{\mathcal P_l}\left(\bigg\{\big| \frac 1 {k_l(n_l+M_l)} (S^{k_l(n_l+M_l)}h -\mathbb E_{{\mathcal P}_{l}}(S^{k_l(n_l+M_l)}h))\big |\geqslant k_l^{-1/2+\eta}\bigg\}\right)\\
&\lesssim \frac{k_l^2M_l^2+k_ls_l^2}  {k_l^{1+2\eta}(n_l+M_l)^2}\leqslant \frac{k_l^{1-2\eta}M_l^2}{n_l^2}+k_l^{1-\eta_l^*-2\eta}\xrightarrow{l\to\infty} 0.
 \end{align*}
This establishes the first part of the theorem.

For the second part, notice that the Berry-Esseen theorem applied to the $k$ blocks in the proof of Theorem \ref{theo:1} yields that the set of points satisfying  (\ref{eq:4.2-5}) is of order $k_l^{-1/2}$, hence the set of $p$ not satisfying (\ref{eq:4.2-5}) has measure zero by  Borel-Cantelli lemma. This implies the second part of the theorem.
\end{proof}

Since the measure of maximal entropy is often mixing it seems to be suggesting that the central limit theorem can only be derived through measures on periodic orbits if these reflect the mixing properties of the measure of maximal entropy. This seems to be difficult (if not at all impossible in general) using the global specification property. This is supported by the fact that periodic points do not have a natural filtration of $\sigma$-fields. Thus, an application of the method in Section \ref{sec:globalCLT} provides a different approach to this problem. At this point we are not discussing this further; instead,
we will derive a CLT of a somewhat different type.

Assume that $(X,T)$ satisfies local specification property with a unique measure of maximal entropy $\mu$. 
In a topologically mixing system with local specification property, one can show in parallel to Theorem \ref{theo:weak-convergence} that $\mu$ is the weak limit of $\nu_{P_n}$, the uniform measure on periodic points.
The following proposition is one from  a variety of others which can be proved along the lines.
We denote by
$\mu_Y$  the induced measure of $\mu$ on $Y$ and call an open, non-empty set $B$ an $\mu$-set if  $\mu(\partial B)=0$.

 \begin{proposition}\label{prop:9.1} Let $(X,T)$ be a $\epsilon^*$-positively expansive and topologically mixing system satisfying the local specification property. Let $\mu$ be the unique measure of maximal entropy.
Consider $\{\mathcal U_l\}_{l\in\mathbb{N}}$ a family   of pairwise disjoint open $\mu$-sets  of diameter less than $\delta(\epsilon^*/3)$ and let $\alpha_l:= \bigvee_{i=0}^{k_l-1} T^{-in_l}\mathcal U_l$, where
$k_l, n_l\in\mathbb N$ and $n_l\geqslant N(\epsilon^*/3)$ (see Definition~\ref{def:specloc}).
Assume
\begin{equation}\label{eq:9.2}
\mu\left(\bigcup_{U\in\mathcal U_l} U\right) = 1- o\left(\frac1{k_l}\right),
\end{equation}
\begin{equation}\label{eq:9.2b}
\lim_{l\to\infty}\sup_{A\in \alpha_l}\limsup_{j\in\mathbb N}\frac {\mu(B_{\epsilon^*}^j(A)\cap P_j)}{\mu(A\cap P_j)}= 1.
\end{equation}

\noindent
Consider a family $\{h_l\}_{l\in\mathbb{N}}$ of real valued functions in $L^3(\mu)$ such that
\begin{enumerate}
\item  for $a_i=(i-1)n_l$ and $s_l^2=\max_{A\in \alpha_l} \sigma^2_{\mu_A} (S^{k_l n_l} h_l)\to\infty$
\begin{equation}\label{eq:9.1}
\lim_{l\to\infty}\frac 1{s_l^2} \max_{A\in\alpha_l}\sum_{1\leqslant i\leqslant k_l} \int_A (\omega_{a_i}^{n_l}(h_l, \epsilon^*, p))^2 d\mu_A(p)=0,
\end{equation}

 \item $\exists K\in \mathbb R$ such that $\forall l\in\mathbb N, A \in \alpha_l$ and $1\leqslant i \leqslant k_l$
 \begin{equation}\label{eq:9.3}
  \mathbb E_{\mu_A}\left(|S^{n_l}_{a_i}h_l - \mathbb E_{\mu_A} (S_{a_i}^{n_l}h_l)|^3\right) \leqslant K  {\sigma}^3_{\mu_A}(S_{a_i}^{n_l}h_l),
  \end{equation}
\item  for some measurable function $\sigma_x$
 \begin{equation}\label{eq:9.5}
 \frac1{s_l^2}\sum_{A\in\alpha_l}\mathbbm 1_A(x)\cdot \sigma^2_{\mu_A}(S^{k_l n_l}h_l)
  \xrightarrow{L^1(\mu)}\sigma_x.
\end{equation}
 \end{enumerate}
 Then for any $t\in\mathbb R$
 $$\lim_{l\to\infty}\mu\big(\big\{ S^{k_l n_l}h_l- \sum_{A\in \alpha_l} \mathbbm 1_A \cdot \mathbb E_{\mu_A} (S^{k_l n_l}h_l) \leqslant t  s_l\big\}\big)=\int_X \mathcal N(t/\sigma_x) d\mu(x),$$
 where for $\sigma_x=0$, $\mathcal N(t/\sigma_x)= \mathbbm 1_{\{s>0\}}(t)$
is the distribution function of the Dirac measure in $0$.
    \end{proposition}

\begin{proof}
Notice  that since by (\ref{eq:9.2})
$$ \mu \left(X\setminus\bigcup_{A\in \alpha_l} A\right)\leqslant \sum_{i=1}^{k_l} \left(1-\mu(T^{-a_i}(\bigcup_{U\in\mathcal U_l}U))\right) \to 0,$$
we can consider the distribution restricted to $\cup_{A\in\alpha_l} A$.
We  first   estimate
\begin{equation}\label{eq:9.6}
 \mu_A\big(\big\{ S^{k_l n_l}h_l(x)- \mathbb E_{\mu_A} (S^{k_l n_l}h_l)\leqslant t s_{l,A}\big\}\big) - \mathcal N(t)
\end{equation}
for  $A\in\alpha_l$,  where
 $$s_{l,A}^2= \sum_{i=1}^{k_l} \sigma^2_{\mu_A}(S^{n_l}_{a_i}h_l).$$
Approximating $h_l$ by a $\mu$-a.e. continuous function and then $\mu$ by the uniform distribution on all periodic points  in ${P}_j$ for sufficiently large $j$, it is sufficient to estimate the expression (\ref{eq:9.6}) with $\mu_A$ replaced by $\nu_{P_{j}\cap A}$.  We also may assume that $j$ is so large that we can approximate
\begin{eqnarray*}
&& \sigma^2_{\mu_A} (S_{a_i}^{n_l}h_l); \qquad \mathbb E_{\mu_A}\left(|S^{n_l}_{a_i}h_{l}- \mathbb E_{\mu_A} (S_{a_i}^{n_l}h_l)|^3\right), \ 1\leqslant i\leqslant k_l
\end{eqnarray*}
by the corresponding expressions under $\nu_{{P}_j\cap A}$ in such a way that it does not affect the limiting distribution and that the assumptions still hold for $\nu_{P_j\cap A}$.

The problem now reduces to proving the statement for the dynamical array $$\{ S^{n_l}_{a_i} h_l-  \mathbb E_{P_j\cap A}(S^{n_l}_{a_i} h_{l})\}_{1\leqslant i\leqslant k_l}$$ on the probability space $(X, \nu_{{P}_j\cap A})$.

By Proposition \ref{prop:locidp} there exists a locally $\epsilon^*/3$-independent set $\mathcal P_l\subset P_j$ with  respect to $A$ such that $ A\cap P_j \subset\mathcal P_l\subset B_{\epsilon^*}^j(A)\cap P_j$ and that $\mathcal P_l$ is bijective with a product set $F_l\subset X^{k_l}$ by $\Phi: F_l\to \mathcal P_l$, which approximates $\mathcal P_l$ as in \eqref{eq:6.2}.
By (\ref{eq:9.2b}) and approximating $\mu$ by $\nu_{P_j}$ we may assume that
$$ 1+o(1)\geqslant \frac{\nu_{P_j}(B_{\epsilon^*}^j(A))}{\nu_{P_j}(A)}\geqslant \frac{|F_l|}{|P_j\cap A|}\geqslant 1.$$
Together with  condition \eqref{eq:9.1}, it follows that
we may exchange the measure $\nu_{P_j\cap A}$  for the product measure $\nu_{F_l}$.
With $\nu_{F_l}$ we can apply the central limit theorem  with rate estimate (see \cite{Pet75}, Chapter V, Section 2, Theorem 3). Since the quotient of the third centered moment of $S^{n_l}_{a_i} h_l\circ \Phi$  divided by its variance raised to the $3/2$ moment stays bounded by (\ref{eq:9.3}) it follows that
$$\left|\nu_{F_l} (\{S^{k_l n_l} h_l\circ \Phi - \mathbb E_{F_l}(S^{k_l n_l} h_{l}\circ \Phi)\leqslant t s_{{l},A}\})- \mathcal N(t)\right|
= O(1/\sqrt{k_l}),$$
where the estimate is uniform over $\mathcal P_l$ and $A\in\alpha_l$.
The same estimate holds therefore for $\nu_{P_j\cap A}$, hence
 \begin{eqnarray*}
 && \nu_{P_j}\left(\bigg\{x\in \bigcup_{A\in \alpha_l} A: S^{k_ln_l} h_l(x)- \sum_{A\in \alpha_l} \mathbbm 1_A(x) \cdot \mathbb E_{\mu_A} (S^{k_ln_l}h_l)\leqslant t s_{l}\bigg\}\right) \\
 &=& \sum_{A\in\alpha_l} \nu_{P_j\cap A}(\{S^{k_ln_l} h_l  - \mathbb E_{\mu_A} (S^{k_ln_l}h_l)\leqslant t s_{l}\}) \nu_{P_j}(A)+o(1)\\
 &=& \sum_{A\in\alpha_l} \mathcal N(ts_{l}/s_{{l},A})  \nu_{P_j}(A)+o(1)\\
&=&  \int \mathcal N(t/\sigma_x) d\nu_{P_j}(x)+o(1),
\end{eqnarray*}
since $ s_{{l},A}/s_{l}$ converges to $\sigma_x$ in $L^1(\mu)$ by assumption (\ref{eq:9.5}). By the choice of $P_j$, the conclusion holds as well for $\mu$.
\end{proof}

The method of proof of Proposition \ref{prop:9.1} has many applications. Here we use  the proposition to obtain convergence to a mixture of normal distributions for  a special class of square integrable functions, which we call wildly oscillating.
We call a finite partition $\beta$ of $X$ a $\nu$-partition if for each $B\in \beta$,  $B\subset\overline{\mbox{int} B}$ and $B$ is a $\nu$-set.

\begin{definition}\label{def:5.2} Let $(X, T, \nu)$ be a probability preserving dynamical system.  Suppose
for every $l\in\mathbb N$, $\beta_l$ is a $\nu$-partition of diameter less than $\epsilon_l$ with $\epsilon_l\to0$ and $W_l\geqslant n_l$ are natural numbers. A square integrable  function $h:X\to \mathbb R$ is wildly oscillating with data $(\beta_l, \epsilon_l, W_l, n_l)_{l\in \mathbb N}$, if
with $k_l:=[W_l/n_l]$, $\alpha_l:=\bigvee_{i=0}^{k_l-1}T^{-i n_l} \beta_{l}$ and $s_l^2:=\max_{A\in\alpha_l} \sigma^2_{\nu_A}(S^{k_l n_l}h)$,
\begin{enumerate}
 \item
 $$s_l\to\infty \quad \text{ and }\quad \frac{1}{s_l^2}\max_{1\leqslant r< n_l}\max_{A\in\alpha_l} \sigma^2_{\nu_A}(S^r h)\to 0,$$
\item
$$\lim_{l\to\infty}\frac 1{s_l^2} \max_{A\in\alpha_l}\sum_{i=1}^{k_l} \int_A \left(\omega_{(i-1)n_l}^{n_l}(h, \epsilon_l, p)\right)^2 d\nu_A(p)=0.$$
\end{enumerate}
\end{definition}

  \begin{theorem}\label{theo:8.2} Let $(X,T)$ be a positively expansive and topologically mixing dynamical system with the local specification property.
Let $\mu$ be the unique measure of maximal entropy. Let $h :X\to \mathbb R$ be a wildly oscillating function in $L^3(\mu)$ with data $(\beta_l,\epsilon_l, W_l, n_l=N(\epsilon_l))_{l\in \mathbb N}$. Assume that with notations above
 \begin{enumerate}
\item $\exists K>0$
 such that $\forall  l\in \mathbb N, A\in \alpha_l, 1\leqslant i\leqslant k_l$
 \begin{equation}\label{eq:9.30}
\mathbb E_{\mu_A}(| S_{(i-1)n_l}^{n_l} h-\mathbb E_{\mu_A}(S^{n_l}_{(i-1)n_l}h)|^3)\leqslant  K (\sigma^2_{\mu_A}(S^{n_l}_{(i-1)n_l} h))^{3/2},
 \end{equation}
\item
\begin{equation}\label{eq:9.32}
\frac1{s_l^2}\sum_{A\in\alpha_l}\mathbbm 1_A\cdot \sigma^2_{\mu_A}(S^{W_l}h) \text{ converges in } L^1.
  \end{equation}
\end{enumerate}
 Then with respect to $\mu$
 \begin{equation}\label{eq:9.36}
  \frac 1{s_l} \left(S^{W_l}h-  W_l \int hd\mu\right)
  \end{equation}
 converges weakly to a mixture of centered  normal distributions as $l\to \infty$.
    \end{theorem}

    \begin{proof} First note that by the assumptions Theorem \ref{theo:weak-convergence} applies so that the measure $\mu$ of maximal entropy is approximable by periodic point measures. The proof now follows applying Proposition \ref{prop:9.1}.
    \end{proof}

 \begin{remark} (1) Condition (\ref{eq:9.32}) of the above theorem may be replaced by the assumption that
 \begin{equation}\label{eq:9.40}
  \frac 1{s_l}  \left(\mathbb E_\mu(S^{W_l} h|\alpha_l)- W_l \int h d\mu\right)
  \end{equation}
 converges to some normal distribution. Moreover, if
 $$\frac{1}{s_l^2}{\sigma^2_\mu(S^{W_l} h|\alpha_l)}\to 1,$$
 then \eqref{eq:9.36} converges to the standard normal distribution.

(2) It seems to be a difficult problem to show asymptotic normality  for the functions in (\ref{eq:9.40}). One needs to prove a CLT under Lindeberg conditions for  locally constant functions. If the dynamics is Gibbs-Markov (\cite{AarDen01}) the CLT in \cite{DenSenZha17} is not applicable, since the Lipschitz norm of the conditional expectation $\mathbb E_\mu(\cdot|\alpha_l)$ generally  is of exponential growth as $l\to \infty$. In this case it seems to be more promising to use concepts of probabilistic mixing.

(3) The conditions of Theorem \ref{theo:8.2} are natural for functions for which one wants to study local fluctuation of ergodic sums. It seems to be quite possible to weaken condition (\ref{eq:9.30}) since a uniform speed of convergence in the Berry-Esseen type theorem  might be more than sufficient to get the result.

   (4) Nontrivial examples of functions which satisfy the assumptions of the theorem
may be constructed as in \cite{BurDen87}.

 \end{remark}

\section{The decomposition theorem for fluctuations}\label{sec:decomposition}

Here we are considering a positively expansive and topologically mixing dynamical system $(X,T)$ with the local specification property. As a consequence, $(X,T)$ also  fulfills the global specification property, and hence the fluctuation of an ergodic sum has at least two competing sources of randomness demonstrated by Theorems \ref{theo:1} and \ref{theo:2}.  Both effects may be present for a given ergodic sum. The global specification determines a global CLT while the local one determines many local CLT's determined by sequences of open sets.
The effect with the fastest growing variances dominates the CLT. If variances have asymptotic equivalent growth rates the resulting limit distribution will be a mixture of Gaussian distributions. In Section \ref{sec:mme} we present such an application.
As a result one needs to study the behavior of the variance of ergodic sums and to decompose it according to  different sources of fluctuation. This is the content of this section.

Let $\epsilon>0$, $\mathcal U$ be a finite collection of open sets and $k,n \in \mathbb N$. Let $\alpha=\bigvee_{i=0}^{k-1}T^{-in}\mathcal U$. Recall that any locally $\epsilon$-independent set $\mathcal P\subset P_{kn}$ is associated to a product set $F_{\mathcal P}\subset X^k$ by a bijection $\Phi$.

\begin{definition}\label{def:var} The variation by periodic orbits over $\mathcal U$  of a  function $h:X\to \mathbb R$ is defined by
$$  \mbox{\rm Var}_{\text{\rm per}}(h)= \min\left\{ \sum_{\mathbf x\in F_{\mathcal P}}\sum_{i=1}^k \left(h(T^{(i-1)n}\Phi(\mathbf x))- h(x_i)\right)^2\right\},$$
where the minimum is taken over all possible choices of $F_{\mathcal P}$  as given above.
\end{definition}
We denote by $\Pi(\mathcal U)$ the collection of all $F_{\mathcal P}=\prod_{i=1}^k E_i$ where the minimum in Definition \ref{def:var} is attained.
\begin{definition}\label{def:local}
The local variation over $\mathcal U$ of a function $h:X\to \mathbb R$ is defined as
$$ \mbox{\rm Var}_{\text{\rm loc}}(h)= \min\left\{\sum_{i=1}^k  \sum_{x_i\in E_i}\left( h(x_i)- \mathbb E_{E_i}(h)\right)^2: F_{\mathcal P}\in \Pi(\mathcal U)\right\}.$$
\end{definition}

We say that $F_{\mathcal P}\in\Pi(\mathcal U)$ is CLT-admissible (for a function $h$) if it minimizes the expression in Definition \ref{def:local} (and \ref{def:var} as well).

\begin{definition}\label{def:Holder}
The H\"older variation over $\mathcal U$ and $F_{\mathcal P}\in \Pi(\mathcal U)$ of a function $h:X\to \mathbb R$ is defined by
$$ \mbox{\rm Var}_{\text{\rm H\"ol}}= |F_{\mathcal P}|\sum_{i=1}^k \left(\mathbb E_{E_i}(h)- \mathbb E_{P_{kn}}(h)\right)^2.$$
\end{definition}

\begin{definition}\label{def:global} The total variation over $\mathcal U$ of a function $h:X\to\mathbb R$ is defined as
$$ \mbox{\rm Var}_{\text{\rm tot}} = \sum_{p\in \mathcal P} \sum_{i=1}^k \left(h(T^{(i-1)n}p) - \mathbb E_{P_{kn}}(h)\right)^2.$$
\end{definition}

\begin{theorem}\label{theo:7.5} Let $h:X\to \mathbb R$ be a function. Then for any CLT-admissible $F_{\mathcal P}\in\Pi(\mathcal U)$  we have
$$ \mbox{\rm Var}_{\text{\rm tot}} = \mbox{\rm Var}_{\text{\rm per}} +\mbox{\rm Var}_{\text{\rm loc}}+ \mbox{\rm Var}_{\text{\rm H\"ol}} + 2 \mbox{\rm Cov}_{\text{\rm per},\text{\rm loc}+ \text{\rm H\"ol}},$$
where
$$\mbox{\rm Cov}_{\text{\rm per},\text{\rm loc}+ \text{\rm H\"ol}}= \sum_{p\in \mathcal P} \sum_{i=1}^k \left(h(T^{(i-1)n}(p)) - h(x_i(p))\right)\left(h(x_i(p))- \mathbb E_{P_{kn}}(h)\right).$$
Moreover,
$$ \mbox{\rm Cov}_{\text{\rm per},\text{\rm loc}+ \text{\rm H\"ol}} \leqslant \sqrt{ \mbox{\rm Var}_{\text{\rm per}}(\mbox{\rm Var}_{\text{\rm loc}}+\mbox{\rm Var}_{\text{\rm H\"ol}})}.$$
\end{theorem}
\begin{proof}
\begin{align*}
\mbox{\rm Var}_{\text{\rm tot}} &= \sum_{p\in \mathcal P}\sum_{i=1}^k \left( h(T^{(i-1)n}(p)) -
h(x_i(p)) +h(x_i(p))- \mathbb E_{E_i}(h)\right.\\
&\left. \quad + \mathbb E_{E_i}(h) - \mathbb E_{P_{kn}}(h)\right)^2\\
&= \mbox{\rm Var}_{\text{\rm per}}+\mbox{\rm Var}_{\text{\rm loc}}+\mbox{\rm Var}_{\text{\rm H\"ol}}\\
&\quad +2\sum_{p\in \mathcal P} \sum_{i=1}^k \left(h(T^{(i-1)n}(p)) - h(x_i(p))\right)\left(h(x_i(p))- \mathbb E_{P_{kn}}(h)\right)\\
&\quad + 2 \sum_{p\in \mathcal P} \sum_{i=1}^k \left(h(x_i(p))- \mathbb E_{E_i}(h)\right)\left(\mathbb E_{E_i}(h)- \mathbb E_{P_{kn}}(h)\right).
\end{align*}
Since the fifth summand vanishes the theorem follows.
\end{proof}

\begin{remark} A similar decomposition of the total variation can be obtained for systems with the global specification property.
\end{remark}

\section*{Acknowledgments}  The research was supported  by {\it n\'umero 158/2012 de Pe\-squisador Visitante Especial de CAPES}. M.D and X.Z also thank the Center of Mathematical Sciences of Huazhong University of Science and Technology for the support when the paper was finalized. S.S was supported by the CNPq. X.Z was supported by PNPD/CAPES and then by Fapesp grant \#2018/15088-4.


\end{document}